\documentclass[12pt]{amsart}
\usepackage[margin=1in]{geometry}

\usepackage{url,amsmath,amssymb,mathrsfs}
\usepackage{enumitem}
\usepackage{commath}
\usepackage{physics}
\usepackage{nicefrac}
\usepackage{amsthm}
\usepackage{xfrac}
\usepackage[all]{xy}
\usepackage{graphicx}
\graphicspath{ {./images/} }

\newtheorem{theorem}{Theorem}[section]
\newtheorem{lemma}[theorem]{Lemma}
\newtheorem{proposition}[theorem]{Proposition}

\newtheorem{conjecture}[theorem]{Conjecture}

\newtheorem{innercustomgeneric}{\customgenericname}
\providecommand{\customgenericname}{}
\newcommand{\newcustomtheorem}[2]{%
	\newenvironment{#1}[1]
	{%
		\renewcommand\customgenericname{#2}%
		\renewcommand\theinnercustomgeneric{##1}%
		\innercustomgeneric
	}
	{\endinnercustomgeneric}
}
\newcustomtheorem{customthm}{Theorem}
\newcustomtheorem{customlemma}{Lemma}

\theoremstyle{definition}
\newtheorem{example}[theorem]{Example}
\newtheorem{definition}[theorem]{Definition}

\newcommand{\kron}[2]{\left(\frac{#1}{#2}\right)}
\newcommand{\bZ}{\mathbb Z}
\newcommand{\bQ}{\mathbb Q}
\newcommand{\bN}{\mathbb N}

\newcommand{\Rbar}{\overline{R}}

\newcommand{\modulo}[1]{\textup{ (mod }#1)}

\newcommand{\ow}{\textup{otherwise}}

\newcommand{\cO}{\mathcal{O}}

\newcommand{\term}[1]{\textbf{\textup{#1}}}

\newcommand{\quot}[2]{\large\sfrac{#1}{#2}}
\newcommand{\bigquot}[2]{\Large\sfrac{#1}{#2}}

\newcommand{\lcm}{\textup{LCM}}

\title[Locally Associated Quadratic Orders]{Locally Associated Orders in Real Quadratic Number Fields}
\author{Grant Moles\and Talha Khan}

\begin{document}

\begin{abstract}
    In 2025, the concept of an order in a number field being associated, ideal-preserving, or locally associated was introduced in order to tackle problems in factorization. In this paper, we explore locally associated orders in real quadratic number fields of the form $\mathbb{Q}[\sqrt{p}]$, with $p\in\mathbb{N}$ prime. In particular, we develop strategies and produce results which make determining when a given order in such a number field is (or is not) locally associated much easier. We also highlight the relatively few cases which defy simple characterization, leading to a conjecture on the solutions to Pell's equations of the form $x^2-y^2p=1$.
\end{abstract}

\maketitle

\section{Introduction}
Fundamental to our understanding of the ring of rational integers $\bZ$ is the aptly named Fundamental Theorem of Arithmetic, which states that every integer (aside from $-1$, $0$, and $1$), factors uniquely (up to reordering and sign) into a product of prime integers. In other words, $\bZ$ is a unique factorization domain (UFD). That said, not every integral domain exhibits such nice factorization; in fact, one does not need to look far beyond $\bZ$ for unique factorization to fail. The canonical example of this failure is the ring $R=\bZ[\sqrt{-5}]=\{a+b\sqrt{-5}|a,b\in\bZ\}$, in which $6$ decomposes into the non-equivalent irreducible factorizations $2\cdot 3$ and $(1+\sqrt{-5})(1-\sqrt{-5})$. In this case, the ring $R$ is in fact a half-factorial domain (HFD); that is, every element decomposes into a unique number of irreducible factors (note that either way one factors 6 in the example above, there are exactly two irreducible factors).

The study of factorization, and in particular the quantification of how badly unique factorization fails, has been an exciting field of study over the last several decades. Half-factorial domains were first described in \cite{carlitz} by Carlitz in 1960; this property was further investigated and named by Zaks in \cite{zaks1} and \cite{zaks2}. In these papers, a complete characterization of half-factorial rings of algebraic integers was obtained. Additional classifications of domains by their factorization type were introduced in \cite{andersonandersonzafrullah}, and more recently in \cite{boyntoncoykendall}.

In 1990, Valenza in \cite{valenza} expanded upon the idea of a half-factorial domain, introducing the concept of elasticity as a measure of how badly unique factorization fails in an atomic domain. This paper (specifically, its Proposition 4), along with Narkiewicz in \cite{narkiewicz}, gave a way to use the ideal class group of a ring of algebraic integers to completely determine its elasticity, expanding on the ideas in \cite{carlitz}.

With the question of elasticity solved in the case of rings of algebraic integers, the natural next step is to consider orders within an algebraic number field. In this paper, we will primarily focus on orders in a quadratic number field, which can be characterized as follows.

\begin{proposition}
    Let $K$ be a quadratic number field, i.e. a ring of the form $K=\bQ[\sqrt{d}]=\{a+b\sqrt{d}|a,b\in\bQ\}$ for some squarefree integer $d$. The \term{ring of algebraic integers} in $K$ is the ring $\cO_K=\bZ[\alpha]$, where $\alpha=\frac{1+\sqrt{d}}{2}$ if $d\equiv 1\modulo{4}$ and $\alpha=\sqrt{d}$ otherwise. An \term{order} in $K$ is any ring of the form $R_n=\bZ[n\alpha]=\{a+nb\sqrt{d}|a,b\in\bZ\}$ for some $n\in\bN$. We will refer to $n$ as the \term{index} of the order $R_n$ in $K$.
\end{proposition}

Throughout this paper, orders in quadratic number fields will be of primary interest. As the ring of algebraic integers $\cO_K$ will always be the integral closure of any order in $K$, we will often write $\Rbar$ to refer to this ring.

Halter-Koch in \cite[Theorem 6]{halter-koch} and Coykendall in \cite[Theorem 2.3]{coykendallhfd} gave the following characterizations of half-factorial orders in quadratic number fields. For convenience, the theorem will be presented in the notation found in this paper rather than those of the original papers; in particular, for a ring $R$, we will use $U(R)$ to denote the group of units in $R$.

\begin{theorem}
    \label{halter}
    Let $K=\bQ[\sqrt{d}]$ be a real quadratic number field (i.e. $d>0$) and $R_n$ the index $n>1$ order in $K$. Then $R_n$ is an HFD if and only if the following properties all hold:
    \begin{enumerate}
        \item $\Rbar$ is an HFD;
        \item $\Rbar=R_n\cdot U(\Rbar)$; and
        \item $n$ is either prime or twice an odd prime.
    \end{enumerate}
\end{theorem}

\begin{theorem}
    \label{imaginaryhfd}
    Let $K=\bQ[\sqrt{d}]$ be an imaginary quadratic number field (i.e. $d<0$) and $R_n$ the index $n>1$ order in $K$. Then $R_n$ is an HFD if and only if $d=-3$ and $n=2$.
\end{theorem}

Halter-Koch's result was expanded upon in 2023 and 2025 in \cite[Theorem 1.1]{rago} and \cite[Theorem 3.8]{radicalconductor}, respectively. In the former, a full characterization of half-factorial orders in a general number field was developed; in the latter, conditions for the elasticity of an order to match that of its ring of algebraic integers were investigated. Interestingly, all three results included the condition that the order $R$ in question had the property that $\Rbar=R\cdot U(\Rbar)$, i.e. that any element $\alpha\in \Rbar$ can be written as the product of an element $r\in R$ and a unit $u\in U(\Rbar)$. This led to the following definitions being presented in \cite[Definitions 2.1.1, 2.1.5, 2.4.1]{dissertation} and later expanded upon in \cite{radicalconductor} and \cite{subringrelations}.

\begin{definition}
    \label{aipla}
    Let $K$ be a number field and $R$ an order in $K$ with conductor ideal $I:=(R:\Rbar)=\{\alpha\in R|\alpha\Rbar\subseteq R\}$. We say that $R$ is an \term{associated order} if $\Rbar=R\cdot U(\Rbar)$. We say that $R$ is a \term{ideal-preserving order} if, for any $\Rbar$-ideals $J_1\nsubseteq J_2$, $R\cap J_1\nsubseteq J_2$. Finally, we say that $R$ is a \term{locally associated order} if
    $$\bigquot{U(\Rbar)}{U(R)}\normalsize\cong \bigquot{U(\quot{\Rbar}{I})}{U(\quot{R}{I})}.$$
\end{definition}

The utility of associated orders has already been demonstrated; this property can give us information about how the known factorization properties of $\Rbar$ influence the factorization properties of $R$. The utility of the other two properties presented here are demonstrated in \cite{subringrelations}. Of particular interest is \cite[Corollary 4.13]{subringrelations}, which demonstrates how these properties relate to one another.

\begin{theorem}
    \label{aiffipla}
    Let $R$ be an order in a number field $K$. Then $R$ is associated if and only if $R$ is both ideal-preserving and locally associated.
\end{theorem}

The purpose of this paper is to expand on the discussion in the closing section of \cite{subringrelations}, in which methods for finding associated, ideal-preserving, and locally associated quadratic orders were developed. In \cite[Theorem 5.8]{subringrelations}, a full characterization of ideal-preserving quadratic orders was produced depending only on the prime ideal factorization in $\Rbar$ of the primes $p$ dividing the index $n$; classically, for odd primes $p$, this can be determined using the Legendre symbol $\kron{d}{p}$. An alternate characterization for locally associated quadratic orders was also produced in \cite[Corollary 6.3]{subringrelations} using the following function.
\begin{definition}
    \label{L function}
    We define the function $L:\bN\times \bZ\to \bN$ as follows:
    \begin{enumerate}
        \item $L(1,d)=1$ for every $d\in\bZ$.
        \item For any $k\in\bN$:$$L(2^k,d)=\begin{cases} 2^{k-1},&d\equiv 1\modulo{8};\\3\cdot 2^{k-1},&d\equiv 5\modulo{8};\\2^k,&\ow.\end{cases}$$
        \item For any $k\in\bN$ and odd prime $p$, $L(p^k,d)=p^{k-1}\left(p-\kron{d}{p}\right)$.
        \item If $m,n\in\bN$ are coprime, then $L(mn,d)=L(m,d)\cdot L(n,d)$.
    \end{enumerate}
    That is, for a fixed $d\in\bZ$, the function $L(\cdot,d):\bN\to\bN$ is a multiplicative arithmetic function.
\end{definition}

\begin{theorem}
    \label{locally associated iff}
    Let $R_n$ be the index $n$ order in the quadratic number field $K=\bQ[\sqrt{d}]$, where $d$ is a squarefree positive integer. Let $u$ be the fundamental unit in $\Rbar$, and let $u^m$ be the minimal power of $u$ lying in $R_n$. Then $m|L(n,d)$. Moreover, $R_n$ is a locally associated order if and only if $m=L(n,d)$.
\end{theorem}

Thus when determining whether the order $R_n$ in $\bQ[\sqrt{d}]$ is locally associated, we need only calculate the function value $L(n,d)$ and determine the minimal power $u^m$ of the fundamental unit $u$ which lies in $R_n$. This was used in \cite[Theorem 6.4]{subringrelations} to produce a full characterization of locally associated orders in non-real quadratic number fields (i.e. those with $d<0$) and to produce a table of orders which are associated, ideal-preserving, or locally associated, which can be found at \cite{quadla}.

This paper seeks to determine when real quadratic number fields are locally associated, particularly in the case when $K=\bQ[\sqrt{p}]$ for a prime $p$. In order to do so, the table at \cite{quadla} was used to determine patterns in the list of locally associated orders. These patterns were then used to make conjectures and ultimately produce the results in this paper. The main result, Theorem \ref{main result}, characterizes many cases when orders in fields of the form $K=\bQ[\sqrt{p}]$ are (or are not) locally associated. This result is as follows.

\begin{customthm}{\ref{main result}}
	Let $R_n$ be the index $n=q^k$ order in the (real) quadratic number field $K=\bQ[\sqrt{p}]$, where $p,q\in\bN$ are prime and $k\in\bN$.

     If any of the following conditions hold, $R_n$ is locally associated.
     \begin{enumerate}
         \item $n=2$ and $p\not\equiv 5\modulo{8}$.
         \item $n=4$ and $p\equiv 1\modulo{8}$.
         \item $n=3$ and $p\not\equiv 3\modulo{4}$.
     \end{enumerate}
     Moreover,
     \begin{enumerate}
         \item[(4)] $R_{p^k}$ is locally associated for every $k\in\bN$ if and only if $R_p$ is locally associated.
         \item[(5)] When $p\equiv 5\modulo{8}$, $R_4$ is locally associated if and only if $R_2$ is locally associated.
         \item[(6)] If $q$ is odd and $p\neq q$, then $R_{q^k}$ is locally associated for every $k\in\bN$ if and only if $R_{q^2}$ is locally associated.
     \end{enumerate}
     If any of the following conditions hold, $R_n$ is NOT locally associated.
     \begin{enumerate}
         \item[(7)] $q$ is odd, $p\neq q$, and either $p\equiv 3\modulo{4}$ or $q\equiv 1\modulo{4}$.
         \item[(8)] $n=4$ and $p\equiv 3\modulo{4}$.
         \item[(9)] $n=8$ and $p$ is odd.
     \end{enumerate}
\end{customthm}

The remainder of this paper is laid out as follows. In the second section, we develop a list of general results and lemmata that will be used later. In the third section, we produce specific results which can be used to determine when an order is locally associated. In the fourth section, we discuss cases that are yet undetermined and their relationship with Pell's equation. In the fifth and final section, we conclude with an exploration of potential future directions for this research.

\section{Preliminary Results}
In the interest of simplifying later proofs, we now collect a number of general results which we will reuse. Since determining when an order is locally associated requires using the fundamental unit in the ring of integers $\Rbar$, we begin by presenting useful results regarding the units in $\Rbar$. 

Recall that for any element $\alpha=a+b\sqrt{d}\in \bQ[\sqrt{d}]$, the norm of $\alpha$ is defined to be $N(\alpha)=a^2-b^2d$. Moreover, $\alpha\in \Rbar$ is a unit if and only if $N(\alpha)=\pm 1$. The following propositions give more information on the norm of the fundamental unit; for detailed proofs of these statements, see \cite[Chapter 11]{introalgnumthe}.

\begin{proposition}
    \label{unit norm 3 mod 4}
    Let $d\in\bZ$ be squarefree and $K=\bQ[\sqrt{d}]$. If $d$ has a prime divisor $p$ satisfying $p\equiv 3\modulo{4}$, then the fundamental unit $u$ in $K$ has norm $N(u)=1$.
\end{proposition}

\begin{proposition}
    \label{unit norm prime 1 mod 4}
    Let $p\equiv 1\modulo{4}$ be prime and $K=\bQ[\sqrt{p}]$. Then the fundamental unit $u$ in $K$ has norm $N(u)=-1$.
\end{proposition}

In certain cases, we may need to be more specific about the form of the fundamental unit than just knowing the norm. The following lemmata address this.

\begin{lemma}
    \label{fundamental unit 3 mod 4}
    Let $p$ be a prime number satisfying $p\equiv 3\modulo 4$. Then the fundamental unit in the number field $K=\bQ[\sqrt{p}]$ is of the form $u=a+b\sqrt{p}$, with $a$ an even integer and $b$ an odd integer.
\end{lemma}

\begin{proof}
    Let $u=a+b\sqrt{p}$ be the fundamental unit in $K$. Since $u$, $-u$, $u^{-1}=N(u)(a-b\sqrt{p})$, and $-u^{-1}=-N(u)(a-b\sqrt{p})$ could all equivalently be chosen as the fundamental unit and one of these elements must necessarily be of the form $c+d\sqrt{p}$ with $c$ and $d$ both positive, we can assume without loss of generality that $a,b\in\bN$. Furthermore, since $p$ is an odd prime and $N(u)=a^2-b^2p=1$ (by Proposition \ref{unit norm 3 mod 4}), then $a$ and $b$ must be of opposite parity. Assume toward a contradiction that $a$ is odd and $b$ is even. Using the above norm equation, $(a+1)(a-1)=a^2-1=b^2p$. Note that $a+1$, $a-1$, and $b$ are all even numbers, so dividing both sides by 4 yields $$\frac{a+1}{2}\cdot \frac{a-1}{2}=\left(\frac{b}{2}\right)^2p.$$ Since $\frac{a+1}{2}$ and $\frac{a-1}{2}$ are integers whose product is divisible by the prime number $p$, then necessarily one of them will be divisible by $p$. We will continue under the assumption that $\frac{a+1}{2}$ is divisible by $p$; the other case will follow in an almost identical manner.

    Dividing both sides of the above equation by $p$ gives $$\frac{a+1}{2p}\cdot\frac{a-1}{2}=\left(\frac{b}{2}\right)^2.$$ Now note that $a+1$ and $a-1$ are integers which differ by 2. In particular, since these are both even numbers, $\gcd(a+1,a-1)=2$. Thus, $\frac{a+1}{2p}$ and $\frac{a-1}{2}$ must be coprime. Since the product of these two coprime integers is a perfect square, then $\frac{a+1}{2p}$ and $\frac{a-1}{2}$ must themselves be perfect squares; that is, $x:=\sqrt{\frac{a-1}{2}}$ and $y:=\sqrt{\frac{a+1}{2p}}$ are both integers. Then $x+y\sqrt{p}\in\Rbar$ with $(x+y\sqrt{p})^2=(x^2+y^2p)+2xy\sqrt{p}=a+b\sqrt{p}=u$. Then $u$ is the square of an element in $\Rbar$, contradicting the fact that $u$ is the fundamental unit in $K$. Then $a$ must be even and $b$ must be odd.
\end{proof}

\begin{lemma}
    \label{fundamental unit 1 mod 8}
    Let $p$ be a prime number satisfying $p\equiv 1\modulo{8}$. Then the fundamental unit in the number field $K=\bQ[\sqrt{p}]$ is of the form $u=a+b\sqrt{p}$, with $4|a$ and $b$ odd.
\end{lemma}

\begin{proof}
    Since $p\equiv 1\modulo{8}$, then $\Rbar=\bZ[\frac{1+\sqrt{p}}{2}]$. Thus, we know that the fundamental unit must be of the form $u=\frac{c+d\sqrt{p}}{2}$ with $c,d\in\bZ$, $c\equiv d\modulo{2}$. Since $N(u)=-1$ by Proposition \ref{unit norm prime 1 mod 4}, $c^2-d^2p=-4$. Considering this equation modulo 8, we note that $c^2-d^2\equiv 4\modulo{8}$. As the only squares modulo $8$ are $0$, $1$, and $4$, we can see by inspection $c$ and $d$ must both be even. Thus, $u=a+b\sqrt{p}$ with $a=\frac{c}{2}\in\bZ$ and $b=\frac{d}{2}\in\bZ$.

    Now $a^2-b^2p=-1$; considering this equation modulo 8, we get $a^2-b^2\equiv -1\modulo{8}$. Again, the only squares modulo 8 are $0$, $1$, and $4$, so the only possibility is that $a^2\equiv 0\modulo{8}$ and $b^2\equiv 1\modulo{8}$. Therefore, $4|a$ and $b$ is odd.
\end{proof}

We now turn our attention to the locally associated property itself. In particular, the following results will help us determine when a given order $R$ is (or is not) locally associated based on our knowledge of other related orders. The first is arguably the most important and comes from \cite[Corollary 5.4]{subringrelations}.

\begin{proposition}
    \label{la in towers}
    Let $R$ be a locally associated order in a number field $K$. Then if $S$ is an intermediate order to $R$, i.e. $R\subseteq S\subseteq \Rbar$, $S$ is also locally associated.
\end{proposition}

For the purposes of this paper, Proposition \ref{la in towers} will allow us to limit our focus. For instance, if we are able to show that the index 2 order $R_2$ in a number field $K$ is not locally associated, then we will not need to concern ourselves with checking $R_4$ or $R_6$ in $K$. Since these orders both contain $R_2$, the contrapositive of Proposition \ref{la in towers} immediately tells us they cannot be locally associated. The next theorem will allow us to limit our focus even further.

\begin{lemma}
     Let $d$ be a squarefree integer, $K=\bQ[\sqrt{d}]$, $R$ an order in $K$, and $u$ the fundamental unit in $K$. If $u^m$ is the minimal power of $u$ lying in $R$ and $u^a\in R$ for some $a\in\bN$, then $m|a$.
\end{lemma}

\begin{proof}
    Let $u^m$ be the minimal power of $u$ lying in $R$; that is, $m$ is the order of the element $u\cdot U(R)$ in the quotient group $\quot{U(\Rbar)}{U(R)}$. The result immediately follows from a well-known property of the order of a group element.
\end{proof}

\begin{lemma}
    \label{lcm power}
    Let $d$ be a positive squarefree integer, $K=\bQ[\sqrt{d}]$, and $u$ the fundamental unit in $K$. Let $m$ and $n$ be coprime positive integers and denote by $R_m$ and $R_n$ the orders of index $m$ and $n$ in $K$, respectively. Then $R_m\cap R_n=R_{mn}$, the index $mn$ order in $K$. Furthermore, if $u^r$ and $u^s$ are the minimal powers of $u$ lying in $R_m$ and $R_n$, respectively, then the minimal power of $u$ lying in $R_{mn}$ is $u^{\textup{LCM}(r,s)}$.
\end{lemma}

\begin{proof}
    Let $\Rbar=\bZ[\alpha]$ be the ring of integers in $K$, with $\alpha=\frac{1+\sqrt{d}}{2}$ if $d\equiv 1\modulo{4}$ and $\alpha=\sqrt{d}$ otherwise. Then $R_m=\bZ[m\alpha]=\{a+b\alpha|a,b\in\bZ, m|b\}$ and $R_n=\bZ[n\alpha]=\{a+b\alpha|a,b\in\bZ,n|b\}$. Then note that an element $a+b\alpha\in\Rbar$ lies in $R_m\cap R_n$ if and only if $m|b$ and $n|b$. Since $m$ and $n$ are relatively prime, then $\lcm(m,n)=\frac{mn}{\gcd(m,n)}=mn$, so $m$ and $n$ both divide $b$ if and only if $mn$ divides $b$. Thus, $R_{mn}=\bZ[mn\alpha]=\{a+b\alpha|a,b\in\bZ,mn|b\}=R_m\cap R_n$.

    Now let $u^r$ and $u^s$ be the minimal powers of the fundamental unit $u$ which lie in $R_m$ and $R_n$, respectively, and let $u^k$ be the minimal power lying in $R_{mn}$. Since $u^k\in R_{mn}=R_m\cap R_n$, the previous lemma tells us that $r|k$ and $s|k$. Then $k$ is a common multiple of $r$ and $s$, so $k\geq \lcm(r,s)$. On the other hand, note that since $u^r\in R_m$ and $\lcm(r,s)$ is a multiple of $r$, then $u^{\lcm(r,s)}\in R_m$. Similarly, $u^{\lcm(r,s)}\in R_n$, so $u^{\lcm(r,s)}\in R_m\cap R_n=R_{mn}$. Then since $k$ is the minimal power of $u$ lying in $R_{mn}$, $k\leq \lcm(r,s)$. Since we have shown both inequalities, we conclude that $k=\lcm(r,s)$.
\end{proof}

\begin{theorem}
    \label{la prime powers}
    Let $d$ be a positive squarefree integer and $K=\bQ[\sqrt{d}]$. If $m,n\in\bN$ are coprime integers, then $R_{mn}$ is locally associated if and only if $R_m$ and $R_n$ are locally associated and $L(m,d)$ and $L(n,d)$ are coprime.
\end{theorem}

\begin{proof}
    Assume that $R_m$ and $R_n$ are locally associated orders and that $L(m,d)$ and $L(n,d)$ are coprime. By Theorem \ref{locally associated iff}, this means that $L(m,d)$ and $L(n,d)$ are the minimal powers of the fundamental unit $u$ lying in $R_m$ and $R_n$, respectively. By Lemma \ref{lcm power}, the minimal power of $u$ lying in $R_{mn}$ is therefore $\lcm(L(m,d),L(n,d))=\frac{L(m,d)\cdot L(n,d)}{\gcd(L(m,d),L(n,d))}=L(m,d)\cdot L(n,d)=L(mn,d)$. Then by Theorem \ref{locally associated iff}, $R_{mn}$ is locally associated.

    Now assume that $R_{mn}$ is locally associated. Since $R_m$ and $R_n$ are orders containing $R_{mn}$, Corollary \ref{la in towers} tells us that $R_m$ and $R_n$ must both be locally associated. By Theorem \ref{locally associated iff}, $L(m,d)$ and $L(n,d)$ are the minimal powers of $u$ lying in $R_m$ and $R_n$, respectively. Again using Lemma \ref{lcm power}, the minimal power of $u$ lying in $R_{mn}$ must be $\lcm(L(m,d),L(n,d))=\frac{L(m,d)\cdot L(n,d)}{\gcd(L(m,d), L(n,d))}$. Since $R_{mn}$ is locally associated, this minimal power must also be $L(mn,d)=L(m,d)\cdot L(n,d)$, so necessarily $\gcd(L(m,d),L(n,d))=1$; that is, $L(m,d)$ and $L(n,d)$ are coprime.
\end{proof}

This theorem will be incredibly useful in determining which orders are locally associated. Given any integer $n>1$, we can first factor $n$ into primes, $n=p_1^{a_1}\dots p_k^{a_k}$. Then, we can check whether the orders of prime-power index $R_{p_i^{a_i}}$ are locally associated. From there, to determine whether $R_n$ is locally associated, we need only check whether the values of $L(p_i^{a_i},d)$ are coprime. Since values of $L(n,d)$ are easy to calculate (and checking whether they are relatively prime is often even easier), we will only need to determine when $R_n$ is locally associated for prime-power $n$.

We conclude this section with modular arithmetic arguments which we will later use to prove the main results of this paper. To start, we present a pair of lemmata regarding binomial coefficients.

\begin{lemma}
    \label{binomial q-1}
    Let $q\in\bN$ be prime. For every $0\leq k\leq q-1$, $\binom{q-1}{k}\equiv (-1)^k\modulo{q}$.
\end{lemma}

\begin{proof}
    First, note that $\binom{q-1}{0}=1$ for any $q$, so the result holds for $k=0$. Now assume toward induction that for some $0\leq k<q-1$, $\binom{q-1}{k}\equiv (-1)^k\modulo{q}$. By definition of the binomial coefficient,
    $$\binom{q-1}{k+1}=\frac{(q-1)!}{(k+1)!\cdot (q-1-k-1)!}=\frac{q-1-k}{k+1}\cdot \frac{(q-1)!}{k!\cdot (q-1-k)!}=\frac{q-(k+1)}{k+1}\cdot \binom{q-1}{k}.$$
    Canceling the denominator and reducing modulo $q$ gives us $$(k+1)\binom{q-1}{k+1}\equiv -(k+1)\binom{q-1}{k}\equiv (k+1)(-1)^{k+1}\modulo{q}.$$
    Finally, since $1\leq k+1\leq q-1$, this element is invertible modulo $q$ and can thus be canceled from both sides of the congruence. Thus, $\binom{q-1}{k+1}\equiv (-1)^{k+1}\modulo{q}$, completing the inductive argument.
\end{proof}

\begin{lemma}
    \label{binomial q+1}
    Let $q\in \bN$ be prime. Then:
    $$\binom{q+1}{k}\equiv \begin{cases}1,&k\in\{0,1,q,q+1\};\\0,&\ow.\end{cases}\modulo{q}$$
\end{lemma}

\begin{proof}
    Recall that for any $q$, $\binom{q+1}{0}=\binom{q+1}{q+1}=1$ and $\binom{q+1}{1}=\binom{q+1}{q}=q+1$. Then if $k\in\{0,1,q,q+1\}$, $\binom{q+1}{k}\equiv 1\modulo{q}$. Otherwise,
    $$\binom{q+1}{k}=\frac{(q+1)\cdot q\cdot (q-1)!}{k!\cdot (q+1-k)!}$$
    Since $q$ is prime and $2\leq k\leq q-1$, note that the numerator in this expression has a factor of $q$ while the denominator does not. Then $\binom{q+1}{k}\equiv 0\modulo{q}$.
\end{proof}

\begin{proposition}
    \label{power l over 2}
    Let $p$ be a prime and $q\neq p$ an odd prime, and let $\alpha=a+b\sqrt{p}$, with $a,b\in\bZ$ relatively prime to $q$. Then the coefficient of $\sqrt{p}$ in the expansion of $\alpha^{\frac{L(q,p)}{2}}$ is divisible by $q$ if and only if $\kron{N(\alpha)}{q}=1$, where $N(\alpha)=a^2-b^2p$, the norm of $\alpha$ in $\bQ[\sqrt{p}]$.
\end{proposition}

\begin{proof}
    Let $x+y\sqrt{p}=(a+b\sqrt{p})^\frac{L(q,p)}{2}$. We can retrieve the coefficient $y$ of $\sqrt{p}$ by first subtracting the conjugate and then dividing by $2\sqrt{p}$: $$y=\frac{(a+b\sqrt{p})^{\frac{L(q,p)}{2}}-(a-b\sqrt{p})^{\frac{L(q,p)}{2}}}{2\sqrt{p}}.$$ We will show that this is divisible by $q$ if and only if $\kron{N(\alpha)}{q} = 1$. To do so, we first multiply both sides by the denominator and square to get 
    \begin{align*}
    4py^2&=(a+b\sqrt{p})^{L(q,p)}+(a-b\sqrt{p})^{L(q,p)} - 2(a+b\sqrt{p})^\frac{L(q,p)}{2}(a-b\sqrt{p})^\frac{L(q,p)}{2}\\
    &=\sum_{i=0}^{L(q,p)}{L(q,p) \choose 2i}a^{L(q,p)-i}(b\sqrt{p})^i + \sum_{i=0}^{L(q,p)}{L(q,p) \choose 2i}a^{L(q,p)-i}(-b\sqrt{p})^i - 2(N(\alpha))^\frac{L(q,p)}{2}
    \end{align*}
    Now in the above summations, note that for odd $i$, the terms in the two summations will cancel; for even $i$, the terms will be identical. Then canceling and combining like terms gives us
    $$4py^2=2\sum_{i=0}^{\frac{L(q,p)}{2}}{L(q,p) \choose 2i}a^{L(q,p)-2i}(b\sqrt{p})^{2i} - 2(N(\alpha))^\frac{L(q,p)}{2}.$$
    Now dividing by 2 and simplifying,
    $$2py^2=\sum_{i=0}^{\frac{L(q,p)}{2}}{L(q,p) \choose 2i}a^{L(q,p)-2i}b^{2i}p^i - N(\alpha)^\frac{L(q,p)}{2}$$
    From here, we split into two cases: when $\kron{p}{q}=1$ (so $L(q,p) = q-1$); and when $\kron{p}{q}=-1$ (so $L(q,p)=q+1$). Since $p\neq q$, we do not need to consider a case when $\kron{p}{q}=0$.
    
    Assume that $\kron{p}{q}=1$. The above identity then becomes
    $$2py^2=\sum_{i=0}^{\frac{q-1}{2}}{q-1 \choose 2i}a^{q-1-2i}b^{2i}p^i - N(\alpha)^\frac{q-1}{2}.$$
    We can now use Lemma \ref{binomial q-1} to note that $\binom{q-1}{2i}\equiv 1\modulo{q}$ for each $i$. Moreover, we can use the fact that $a$ is relatively prime to $q$ and Fermat's Little Theorem to note that $a^{q-1-2i}\equiv a^{-2i}\modulo{2i}$ (here, we use the shorthand $a^{-k}$ to denote raising the inverse of $a$ modulo $q$ to the $k^{th}$ power). This gives
    $$2py^2\equiv\sum_{i=0}^{\frac{q-1}{2}}a^{-2i}b^{2i}p^i - N(\alpha)^\frac{q-1}{2}\equiv\sum_{i=0}^{\frac{q-1}{2}}(a^{-2}b^2p)^i-N(\alpha)^{\frac{q-1}{2}}\modulo{q}.$$
    
    Now if $N(\alpha)=a^2-b^2p\equiv 0\modulo{q}$, note that $a^{-2}b^2p\equiv 1\modulo{q}$. In this case, the above congruence yields $2py^2\equiv \frac{q-1}{2}\not\equiv 0\modulo{q}$. Then when $\kron{N(\alpha)}{q}=0$, $y\not\equiv 0$, as desired. We continue under the assumption that $q\nmid N(\alpha)$ and make use of the formula for geometric sums.
    $$2py^2\equiv\frac{1-(a^{-2}b^2p)^\frac{q+1}{2}}{1-a^{-2}b^2p} - N(\alpha)^\frac{q-1}{2}\modulo{q}$$
    We now multiply both sides by $N(\alpha)\equiv a^2(1-a^{-2}b^2p)\modulo{q}$ to give:
    \begin{align*}
    2py^2N(\alpha)&\equiv a^2-a^2(a^{-2}b^2p)^\frac{q+1}{2}-N(\alpha)^\frac{q+1}{2}\\
    &\equiv a^2-a^{-(q-1)}b^{q+1}p^\frac{q+1}{2}-N(\alpha)^\frac{q+1}{2}\\
    &\equiv a^2-b^2p-N(\alpha)^\frac{q+1}{2}\\
    &\equiv N(\alpha)(1-N(\alpha)^{\frac{q-1}{2}})\modulo{q}
    \end{align*}
    
    Now since $q\nmid N(\alpha)$, we cancel this term to get $2py^2\equiv 1-N(\alpha)^{\frac{q-1}{2}}\modulo{q}$. Note that since $q\neq p$ is an odd prime, the left-hand side of this congruence is 0 modulo $q$ if and only if $q|y$. The right-hand side of this congruence is 0 modulo $q$ if and only if $\kron{N(\alpha)}{q}\equiv N(\alpha)^{\frac{q-1}{2}}\equiv 1\modulo{q}$. Thus, when $\kron{p}{q}=1$, $q|y$ if and only if $\kron{N(\alpha)}{q}=1$, as desired.
    
    Now assume that $\kron{p}{q}=-1$. Note in this case that if $q|N(\alpha)$, then $a^2\equiv b^2p\modulo{q}$. Since $b$ is invertible modulo $q$, this means that $p\equiv (ab^{-1})^2\modulo{q}$, and thus $p$ is a quadratic reside modulo $q$. This contradicts the fact that $\kron{p}{q}=-1$, so we conclude that $q\nmid N(\alpha)$.

    Our identity now becomes
    $$2py^2=\sum_{i=0}^{\frac{q+1}{2}}{q+1 \choose 2i}a^{q+1-2i}b^{2i}p^i - N(\alpha)^\frac{q+1}{2}.$$
    By Lemma \ref{binomial q+1}, $\binom{q+1}{2i}\equiv 0\modulo{q}$ for each $1\leq i\leq \frac{q-1}{2}$. Also note that $p^{\frac{q-1}{2}}\equiv \kron{p}{q}\equiv -1\modulo{q}$. Then
    $$2py^2\equiv a^{q+1}+b^{q+1}p^\frac{q+1}{2}- N(\alpha)^\frac{q+1}{2}\equiv a^2-b^2p- N(\alpha)^\frac{q+1}{2}\equiv N(\alpha)(1-N(\alpha)^{\frac{q-1}{2}}) \modulo{q}.$$
    
    As in the previous case, the left-hand side of this congruence is 0 modulo $q$ if and only if $q|y$. Since $q\nmid N(\alpha)$, the right-hand side of this congruence is 0 modulo $q$ if and only if $\kron{N(\alpha)}{q}\equiv N(\alpha)^{\frac{q-1}{2}}\equiv 1\modulo{q}$. Thus, when $\kron{p}{q}=-1$, $q|y$ if and only if $\kron{N(\alpha)}{q}=1$, as desired.
\end{proof}

\begin{proposition}
    \label{p power}
    Let $p$ and $q$ be primes, and let $\alpha=a+b\sqrt{p}$, with $a,b\in\bZ$, $q\nmid a$, and $q^k$ exactly dividing $b$ for some $k\in\bN$ (that is, $q^k|b$ but $q^{k+1}\nmid b$). Then the coefficient of $\sqrt{p}$ in the expansion of $\alpha^q$ is exactly divisible by $q^{k+1}$.
\end{proposition}

\begin{proof}
We will show that the coefficient of $\sqrt{p}$ in the expansion of $\alpha^q$ is divisible by $q^{k+1}$ but not by $q^{k+2}$. Expanding $\alpha^q$ yields
$$\alpha^q=(a+b\sqrt{p})^q=\sum_{i=0}^q\binom{q}{i}a^{q-i}(b\sqrt{p})^i.$$
Since we are only interested in the coefficient of $\sqrt{p}$ in this expansion, we only need to consider the odd terms in this sum. That is, letting $\alpha^q=x+y\sqrt{p}$,
$$y=\sum_{i=0}^{\frac{q-1}{2}} \binom{q}{2i+1}a^{q-2i-1}b^{2i+1}p^i.$$ 

Now note that for any $j\geq 3$, we have $jk\geq 3k\geq k+2$. Since $q^k$ divides $b$, then $q^{k+2}$ divides $b^j$ for any $j\geq 3$. Then modulo $q^{k+2}$, every term in the above sum except the first will vanish; that is,
$$y\equiv qa^{q-1}b\modulo{q^{k+2}}.$$
Since $q^k$ exactly divides $b$ and $q\nmid a$, this term is divisible by $q^{k+1}$ but not by $q^{k+2}$, as desired.

\end{proof}

\section{Determining Locally Associated Orders}
We now turn our attention to actually determining when orders are (or are not) locally associated. Recall from Theorem \ref{locally associated iff} that to determine when the index $n$ order $R_n$ in the quadratic number field $K=\bQ[\sqrt{d}]$ is locally associated, we need to determine the function value $L(n,d)$ and the minimal power $m$ of the fundamental unit $u$ in $\Rbar$ such that $u^m\in R_n$. For the purposes of this paper, we will focus primarily on the case when $d$ is prime. Also recall from our discussion in the previous section that we can limit our attention to orders whose index $n$ is a power of a prime. Finally, recall that when searching for the minimal power $u^m\in R_n$, Theorem \ref{locally associated iff} tells us that we only need to check values of $m$ which divide $L(n,d)$.

 \begin{theorem}
    \label{main result}
     Let $R_n$ be the index $n=q^k$ order in the (real) quadratic number field $K=\bQ[\sqrt{p}]$, where $p,q\in\bN$ are prime and $k\in\bN$.

     If any of the following conditions hold, $R_n$ is locally associated.
     \begin{enumerate}
         \item $n=2$ and $p\not\equiv 5\modulo{8}$.
         \item $n=4$ and $p\equiv 1\modulo{8}$.
         \item $n=3$ and $p\not\equiv 3\modulo{4}$.
     \end{enumerate}
     Moreover,
     \begin{enumerate}
         \item[(4)] $R_{p^k}$ is locally associated for every $k\in\bN$ if and only if $R_p$ is locally associated.
         \item[(5)] When $p\equiv 5\modulo{8}$, $R_4$ is locally associated if and only if $R_2$ is locally associated.
         \item[(6)] If $q$ is odd and $p\neq q$, then $R_{q^k}$ is locally associated for every $k\in\bN$ if and only if $R_{q^2}$ is locally associated.
     \end{enumerate}
     If any of the following conditions hold, $R_n$ is NOT locally associated.
     \begin{enumerate}
         \item[(7)] $q$ is odd, $p\neq q$, and either $p\equiv 3\modulo{4}$ or $q\equiv 1\modulo{4}$.
         \item[(8)] $n=4$ and $p\equiv 3\modulo{4}$.
         \item[(9)] $n=8$ and $p$ is odd.
     \end{enumerate}
 \end{theorem}

 Throughout the proofs that follow, we will let $u\in U(\Rbar)$ denote the fundamental unit in $\Rbar$ and $m\in\bN$ denote the minimal power of $u$ such that $u^m\in R_n$.

    \begin{proof}[Proof of (1).]
        We consider three cases. If $p=2$, then we can verify using the table at \cite{quadla} (or by direct calculation) that $R_2$ is locally associated in $\bQ[\sqrt{2}]$. If $p\equiv 1\modulo{8}$, then $L(2,p)=1$, meaning that $R_2$ is trivially locally associated. Finally, if $p\equiv 3\modulo{4}$, then $L(2,p)=2$ and thus $m$ must be either 1 or 2. Since Lemma \ref{fundamental unit 3 mod 4} tells us that $u=a+b\sqrt{p}$ with $b$ odd, then $m\neq 1$ and thus $R_2$ is locally associated.
    \end{proof}
        \begin{proof}[Proof of (2).]
            Note that $L(4,p)=2$, so $m$ must either be 1 or 2. By Lemma \ref{fundamental unit 1 mod 8}, our fundamental unit is of the form $u=a+b\sqrt{d}$ with $b$ odd. Then $u\notin R_4=\bZ[2\sqrt{p}]$, so $m=2$. Then $R_4$ is locally associated.
        \end{proof}
        \begin{proof}[Proof of (3).]
            If $p=2$, then we can verify using the table at \cite{quadla} (or by direct calculation) that $R_3$ is locally associated in $\bQ[\sqrt{p}]$. Otherwise, $p\equiv 1\modulo{4}$, so we note by Proposition \ref{unit norm prime 1 mod 4} that $N(\alpha)=-1$. Then writing $u=\frac{c+d\sqrt{p}}{2}$, we get that $c^2-d^2p=-4$. If $3|d$, then $c^2\equiv 2\modulo{3}$, a contradiction; then $u\notin R_3$.

            Since $p\neq 3$, we have two cases to consider. If $\kron{p}{3}=1$ (i.e. if $p\equiv 1\modulo{3}$), then $L(3,p)=2$, and thus $m$ must be either 1 or 2. Since we have already shown that $u\notin R_3$, then $m=2$ and thus $R_3$ is locally associated. If $\kron{p}{3}=-1$ (i.e. if $p\equiv -1\modulo{3}$), then $L(3,p)=4$, and thus $m$ must be 1, 2, or 4. We have already established that $m\neq 1$. Now using the fact that $p\equiv -1\modulo{3}$, we note that $c^2-d^2p\equiv c^2+d^2\equiv 2\modulo{3}$, and thus $c^2\equiv d^2\equiv 1\modulo{3}$. Thus, neither $c$ nor $d$ is divisible by 3. We can now use Proposition \ref{power l over 2} to conclude that $(2u)^\frac{4}{2}=(2u)^2\notin R_3$, since $\kron{N(2u)}{3}=\kron{-4}{3}=\kron{2}{3}=-1$. Since $2\nmid n$, then $(2u)^2\notin R_3$ means that $u^2\notin R_3$ and thus $m\neq 2$. Then $m=4$, so $R_3$ is locally associated.
        \end{proof}
          \begin{proof}[Proof of (4).]
             The forward direction of this proof is trivial; if $R_{p^k}$ is locally associated for every $k\in\bN$, then $R_p$ is locally associated.
             
             For the reverse direction, we assume that $R_p$ is locally associated. Since $L(p,p)=p$, this simply means that $u\notin R_p$. To verify that $R_{p^2}$ is locally associated, we note that $L(p^2,p)=p^2$. Since $m$ must divide $p^2$, we need to show that $u^p\notin R_{p^2}$. For $p=2$ and $p=3$, we use the table at \cite{quadla} to verify that $R_4$ and $R_9$ are locally associated in $\bQ[\sqrt{2}]$ and $\bQ[\sqrt{3}]$, respectively. Then we will assume that $p>3$.

             Let $u=\frac{c+d\sqrt{p}}{2}$; since $p\neq 2$, we note that for any $n,k\in\bN$, $u^n\in R_{p^k}$ if and only if $(2u)^n=(c+d\sqrt{p})^n\in R_{p^k}$. Then since $u\notin R_p$, we know that $p\nmid d$. Moreover, since $u^2\notin R_p$, we know that $(2u)^2=(c^2+d^2p)+2cd\sqrt{p}\notin R_p$. Therefore, $p\nmid c$. Now write $(2u)^p=(c+d\sqrt{p})^p=x+y\sqrt{p}$. Expanding the binomial and collecting the coefficients of $\sqrt{p}$ from the odd terms, we get
             $$y=\sum_{i=0}^{\frac{p-1}{2}}\binom{p}{2i+1}a^{q-2i-1}b^{2i+1}p^i.$$
             For $i\geq 2$, note that the terms in the above sum have a factor of $p^2$. Since $p\neq 3$, $p|\binom{p}{3}$, and thus the $i=1$ term in the above sum is divisible by $p^2$ as well. Then
             $$y\equiv pa^{q-1}b\not\equiv 0\modulo{p^2}.$$
             Then $u^p\notin R_{p^2}$, and thus $R_{p^2}$ is locally associated.

             Now assume toward induction that for some $k\geq 2$, $R_{p^k}$ is locally associated; since $L(p^k,p)=p^k$, this means that $\alpha:=u^{p^{k-1}}\notin R_{p^k}$. We will show that $R_{p^{k+1}}$ is locally associated by showing that $u^{p^k}=\alpha^p\notin R_{p^{k+1}}$. Since $\alpha$ is an element of $R_{p^{k-1}}$ (as $L(p^{k-1},p)=p^{k-1}$) but not of $R_{p^k}$, we write $\alpha=\frac{c+d\sqrt{p}}{2}$, with $d$ exactly divisible by $p^{k-1}$. Since $\alpha$ is a unit in $\Rbar$, we know that $p\nmid \alpha$; thus, $p\nmid c$. Then $2\alpha=c+d\sqrt{p}$ has $p\nmid c$ and $d$ exactly divisible by $p^{k-1}$ with $k\geq 2$, so Proposition \ref{p power} tells us that the coefficient of $\sqrt{p}$ in the expansion of $(2\alpha)^p$ is exactly divisible by $p^k$. Then $\alpha^p=u^{p^k}$ lies in $R_{p^k}$ but not in $R_{p^{k+1}}$, and thus $R_{p^{k+1}}$ is locally associated. By induction, $R_{p^k}$ is locally associated for every $k\in\bN$.
         \end{proof}
         \begin{proof}[Proof of (5).]
             Again, the forward direction is trivial. Since $R_4\subseteq R_2$, then Proposition \ref{la in towers} tells us that if $R_4$ is locally associated, $R_2$ must be locally associated as well.
             
             For the reverse direction, assume that $R_2$ is locally associated. Since $L(2,p)=3$, this means that $u=\frac{c+d\sqrt{p}}{2}\notin R_2$. Then $d$ is odd, and since $c$ and $d$ must have the same parity (as $u\in \Rbar$), $c$ is odd as well. To show that $R_4$ is locally associated as well, note that $L(4,p)=6$. Since neither $u$ nor $u^2$ lies in $R_2$, these elements certainly do not lie in $R_4$. Thus, it will suffice to show that $m\neq 3$, i.e. $u^3\notin R_4$. Expanding $u^3$ gives
             $$u^3=\left(\frac{c+d\sqrt{p}}{2}\right)^3=\frac{c^3+3cd^2p}{8}+\frac{3c^2d+d^3p}{8}\sqrt{p}.$$
             We now consider the numerator of the coefficient of $\sqrt{p}$, $3c^2d+d^3p=d(3c^2+d^2p)$. By Proposition \ref{unit norm prime 1 mod 4}, $N(u)=-1$; thus, $N(2u)=c^2-d^2p=-4$, so $d^2p=c^2+4$. Substituting into the numerator yields $d(3c^2+d^2p)=d(4c^2+4)=4d(c^2+1)$. Now since $c$ is odd, we note that $c^2\equiv 1\modulo{4}$. Then $4d(c^2+1)\equiv 8\modulo{16}$, and thus the coefficient of $\sqrt{p}$ in the above expression, $\frac{3c^2d+d^3p}{8}$, is an odd integer. Then $u^3\notin R_4=\bZ[2\sqrt{p}]$, so $R_4$ is locally associated.
         \end{proof}
         \begin{proof}[Proof of (6).]
             The forward direction of this proof is again trivial; if $R_{q^k}$ is locally associated for every $k\in \bN$, then $R_{q^2}$ is of course locally associated.

             Now for the reverse direction, assume that $R_{q^2}$ is locally associated (and thus $R_q$ is locally associated as well). Then $u^{L(q,p)}\in R_q\backslash R_{q^2}$; that is, $u^{L(q,p)}=\frac{c+d\sqrt{p}}{2}$, with $q$ exactly dividing $d$. We proceed in much the same way as in the inductive step of (4). 
             
             Assume toward induction that for some $k\geq 2$, $R_{q^k}$ is locally associated. Since $L(q^k,p)=q\cdot L(q^{k-1},p)$, this means that $\alpha:=u^{L(q^{k-1},p)}\in R_{q^{k-1}}\backslash R_{q^k}$. We will show that $R_{q^{k+1}}$ is also locally associated by showing that $u^{L(q^k,p)}=\alpha^q\notin R_{q^{k+1}}$. We write $\alpha=\frac{c+d\sqrt{p}}{2}$ and note that $q^{k-1}$ exactly divides $d$. Moreover, since $\alpha$ is a unit, $q\nmid \alpha$ and thus $q\nmid c$. Then by Proposition \ref{p power}, the coefficient of $\sqrt{p}$ in the expansion of $(2\alpha)^q=(c+d\sqrt{p})^q$ is exactly divisible by $q^k$. Therefore, $(2\alpha)^q$ (and thus $\alpha^q=u^{L(q^k,p)}$, since $q\neq 2$) lies in $R_{q^k}$ but not $R_{q^{k+1}}$. Then $R_{q^{k+1}}$ is locally associated. By induction, $R_{q^k}$ is locally associated for every $k\in\bN$.
         \end{proof}
         \begin{proof}[Proof of (7).]
             We will show that $R_q$ is not locally associated; this will also show that $R_n$ is not locally associated for every $n=q^k$, $k\in\bN$ by Theorem \ref{la in towers}. First, assume that $p\equiv 3\modulo{4}$. Then $u=a+b\sqrt{p}$ with $a,b\in\bZ$. Since $L(q,p)=q-\kron{p}{q}>1$, then if $q|b$ (i.e. $u\in R_q$), $R_q$ automatically fails to be locally associated. By inspection, $L(q,p)=2$ if and only if $q=3$ and $p\equiv 1\modulo{3}$. In this case, note that if $q|a$, then $1=N(u)=a^2-b^2p\equiv-b^2\modulo{p}$, a contradiction. Otherwise, $L(q,p)>2$. Then if $q|a$, we observe that $u^2=(a^2+b^2p)+2ab\sqrt{p}\in R_q$ and thus $R_q$ is again not locally associated. Then we can assume that $q\nmid a$ and $q\nmid b$. Applying Proposition \ref{power l over 2} tells us that $u^{\frac{L(q,p)}{2}}\in R_q$ since $\kron{N(u)}{q}=\kron{1}{q}=1$. Then $R_q$ is not locally associated.

             Now assume that $q\equiv 1\modulo{4}$ and $p\not\equiv 3\modulo{4}$. We write $u=\frac{c+d\sqrt{p}}{2}$ with $c,d\in\bZ$, $c\equiv d\modulo{2}$. Since $L(q,p)=q-\kron{p}{q}>2$, we note by largely the same argument as above that if $q|c$ or $q|d$, then $R_q$ is not locally associated. Otherwise, $2u=c+d\sqrt{p}$ has both $c$ and $d$ relatively prime to $q$ and we can again apply Proposition \ref{power l over 2}. Since $q\equiv 1\modulo{4}$, we note that $\kron{N(2u)}{q}=\kron{-4}{q}=\kron{4}{q}\cdot\kron{-1}{q}=1$. Then $(2u)^{\frac{L(q,p)}{2}}\in R_q$, and since $q\neq 2$, $u^{\frac{L(q,p)}{2}}\in R_q$. Once again, we conclude that $R_q$ is not locally associated.
         \end{proof}
         \begin{proof}[Proof of (8).]
             We first note that $L(4,p) = 4$. Then by Lemma $\ref{fundamental unit 3 mod 4}$, the fundamental unit is of the form $u=a+b\sqrt{p}$, with $a$ even and $b$ odd. Then $u^2=(a^2+b^2p)+2ab\sqrt{p}\in R_4$, so $R_4$ is not locally associated.
         \end{proof}
         \begin{proof}[Proof of (9).]
            We consider three cases. If $p\equiv 3\modulo{4}$, then the previous case tells us that $R_4$ is not locally associated. Then by Theorem \ref{la in towers}, $R_8$ is certainly not locally associated. If $p\equiv 1\modulo{8}$, then $L(8,p)=4$. Lemma \ref{fundamental unit 1 mod 8} tells us that $u=a+b\sqrt{p}$, with $4|a$ and $b$ odd. Then $u^2=(a^2+b^2p)+2ab\sqrt{p}\in R_8=\bZ[4\sqrt{p}]$, so $R_8$ is not locally associated.
             
            The final case to consider is when $p\equiv 5\modulo{8}$. In this case, $L(8,p)=12$. If $R_2$ is not locally associated, then $R_8$ is certainly not locally associated; we can therefore assume that $R_2$ is locally associated. From the proof of case (5) above, we recall that $u^3=a+b\sqrt{p}$, with $b$ an odd integer. Since $N(u^3)=a^2-b^2p=-1$, then note that $a^2\equiv 1+b^2p\equiv 0\modulo{2}$. Then $a$ must be an even integer, so $u^6=(u^3)^2=(a+b^2p)+2ab\sqrt{p}\in R_8=\bZ[4\sqrt{2}]$. Thus, $R_8$ is not locally associated.
         \end{proof}

 In many cases, these results will allow us to determine when an order $R_n$ is a quadratic number field $K=\bQ[\sqrt{p}]$ is (or is not) locally associated with minimal work required. We conclude this section with some concrete examples.

 \begin{example}
     Let $K=\bQ[\sqrt{3}]$ and $R_n=\bZ[n\sqrt{3}]$ be the index $n$ order in $K$ for $n=13122=2\cdot 3^8$. Note that $L(2,3)=2$ and $L(3^8,3)=3^8$ are relatively prime; then by Theorem \ref{la prime powers}, $R_n$ is locally associated if and only if $R_2$ and $R_{3^8}$ are locally associated. Case (1) above guarantees that $R_2$ is locally associated, and case (4) above guarantees that $R_{3^8}$ is locally associated if and only if $R_3$ is locally associated. Using the table at \cite{quadla} (or by direct calculation), we can observe that $R_3$ is locally associated; thus, $R_n$ is locally associated. Note that this process is much simpler than directly verifying that the minimal power $m$ of the fundamental unit $u=2+\sqrt{3}$ which lies in $R_n$ is $u^{L(n,p)}=(2+\sqrt{3})^{13122}$
 \end{example}

 \begin{example}
     Let $K=\bQ[\sqrt{71}]$ and $R_n=\bZ[n\frac{1+\sqrt{71}}{2}]$ be the index $n$ order in $K$ for $n=1868059634=2\cdot 7^4\cdot 73^3$. Note that $L(2,73)=1$, $L(7^4,73)=7^3\cdot 8$, and $L(73^3,73)=73^3$ are relatively prime; then by Theorem \ref{la prime powers}, $R_n$ is locally associated if and only if $R_2$, $R_{7^4}$, and $R_{73^3}$ are locally associated. Cases (1) and (4) can be used to verify that $R_2$ and $R_{73^3}$ are locally associated as in the previous example. Case (6) guarantees that $R_{7^4}$ is locally associated is and only if $R_{7^2}=R_{49}$ is locally associated. However, the table at \cite{quadla} tells us that $R_{49}$ is not locally associated, and thus $R_n$ is not locally associated.
 \end{example}

 \begin{example}
     Let $K=\bQ[\sqrt{13}]$ and $R_n=\bZ[n\frac{1+\sqrt{13}}{2}]$ be the index $n$ order in $K$ for $n=1965641640625=5^7\cdot 13^2\cdot 53^3$. Note that $L(5^7,13)=5^6(5-\kron{13}{5})$ and $L(53^3,13)=53^2(53-\kron{13}{53})$ are both even and are thus not relatively prime. Then by Theorem \ref{la prime powers}, $R_n$ is not locally associated.
 \end{example}

\section{Undetermined Cases}
The major results from the previous section serve to eliminate most of the work in determining when an order is (or is not) locally associated. That said, there are some specific cases which are either difficult to prove in generality, exhibit unpredictable behavior, or defy simple characterization. In this section, we will explore these cases and outline yet open problems related to them. In addition, we will make frequent reference to the table at \cite{quadla} in order to provide data about the frequency of locally associated orders.

The first case we will consider is the index $p$ order $R_p$ in the quadratic field $K=\bQ[\sqrt{p}]$, where $p$ is a prime. Note in this case that $L(p,p)=p$, so the minimal power of the fundamental unit $u$ which lies in $R_p$ must be either $u$ or $u^p$. In other words, $R_p$ is locally associated if and only if the fundamental unit $u$ does not lie in $R_p$; that is, if and only if $u=a+b\alpha$, with $p\nmid b$.

To further explore this case, recall that when $p\equiv 1\modulo{4}$, $\Rbar=\{\frac{a+b\sqrt{p}}{2}|a,b\in \bZ,a\equiv b\modulo{2}\}$; otherwise, $\Rbar=\{a+b\sqrt{p}|a,b\in\bZ\}$. Since units in $\Rbar$ are characterized by having norm $\pm 1$ (and keeping in mind Propositions \ref{unit norm 3 mod 4} and \ref{unit norm prime 1 mod 4}), the fundamental unit in the first case is $u=\frac{a+b\sqrt{p}}{2}$, where $(a,b)$ is the minimal solution to the Diophantine equation $x^2-y^2p=-4$. In the second case, the fundamental unit is either $1+\sqrt{2}$ if $p=2$ or $a+b\sqrt{p}$, where $(a,b)$ is the minimal solution to the Diophantine equation (in particular, Pell's equation) $x^2-y^2p=1$. This leads us to the following result.

\begin{theorem}
    Let $p$ be an odd prime. Then the index $p$ order $R_p$ in the quadratic number field $K=\bQ[\sqrt{p}]$ is locally associated if and only if the Pell's equation $x^2-y^2p=1$ has an integer solution $(a,b)$ with $p\nmid b$.
\end{theorem}

\begin{proof}
    First, assume that $p\equiv 3\modulo{4}$. Then by Proposition \ref{unit norm 3 mod 4}, the fundamental unit $u=a+b\sqrt{p}$ in $\Rbar$ has norm 1, meaning that $a^2-b^2p=1$. If $R_p$ is locally associated, then $p\nmid b$, and thus the Pell's equation $x^2-y^2p=1$ has an integer solution $(a,b)$ with $p\nmid b$. If $R_p$ is not locally associated, then $u\in R_p$, meaning that any unit $v=c+d\sqrt{p}\in U(\Rbar)$ is actually an element of $R_p$ (since $v=\pm u^k$ for some $k\in\bZ$). Then since the units in $U(\Rbar)$ are in one-to-one correspondence with the integer solutions to the Pell's equation $x^2-y^2p=1$, any solution $(c,d)$ to this equation will  necessarily have $p|d$. This completes the proof when $p\equiv 3\modulo{4}$.

    Now assume that $p\equiv 1\modulo{4}$. Then by Proposition \ref{unit norm prime 1 mod 4}, the fundamental unit $u=\frac{a+b\sqrt{p}}{2}$ in $\Rbar$ has norm $-1$, meaning that $a^2-b^2p=-4$. As before, note that if $R_p$ is not locally associated, then $u\in R_p$, meaning that any unit $v\in U(\Rbar)$ also lies in $R_p$. Since any solution $(c,d)$ to the Pell's equation $x^2-y^2p=1$ necessarily corresponds to a unit $v=c+d\sqrt{p}$ with $N(v)=1$, then necessarily $v\in R_p=\{\frac{a+b\sqrt{p}}{2}|a,b\in\bZ, a\equiv b\modulo{2}, p|b\}$. Thus, $p|2d$; since $p$ is an odd prime, then $p|d$. Now assume that $R_p$ is locally associated, i.e. $u^p$ is the minimal power of $u$ lying in $R_p$. By Definition \ref{L function}, note that $L(2,p)$ must either be 1 (if $p\equiv 1\modulo{8}$) or 3 (if $p\equiv 5\modulo{8}$); in either case, $u^3\in R_2=\bZ[\sqrt{p}]$. Thus, $u^6\in R_2$ as well, so $u^6=c+d\sqrt{p}$ with $c,d\in\bZ$ and $N(u^6)=(N(u))^6=(-1)^6=1$. Thus, $c^2-d^2p=1$. Moreover, note that since $p$ is the minimal power of $u$ lying in $R_p$ and $p\nmid 6$ (as $p\equiv 1\modulo{4}$), then $u^6\notin R_p$. In particular, this means that $p\nmid d$, yielding a solution $(c,d)$ to the Pell's equation $x^2-y^2p=1$ with $p\nmid d$, as desired.
\end{proof}

This theorem tells us that finding certain solutions to Pell's equation is equivalent to showing that the index $p$ order in the number field $\bQ[\sqrt{p}]$ is locally associated, where $p$ is an odd prime. Since the index $2$ order in $\bQ[\sqrt{2}]$ can easily be verified to be locally associated (for instance, by referencing the table at \cite{quadla}), this leads to the following conjecture.

\begin{conjecture}
    \label{pell's equation}
    Let $p\in\bN$ be prime. Then the index $p$ order $R_p$ in the quadratic number field $\bQ[\sqrt{p}]$ is locally associated. Equivalently, for any odd prime $p$, the Pell's equation $x^2-y^2p=1$ has a solution $(a,b)$ with $p\nmid b$.
\end{conjecture}

This conjecture seems relatively simple, but it as yet seems to defy simple proof. That said, the table at \cite{quadla} shows that this conjecture holds for the first 168 primes (all primes $p<1000$).

While this case is perhaps the most interesting unsettled case, there are a number of additional cases that the main results from the previous section do not address. We list them here, along with data from the table at \cite{quadla} regarding how often these orders are (or are not) locally associated. Throughout the following list, we use $n=q^k$ to refer to the prime-power index of the order $R_n$ in the number field $\bQ[\sqrt{p}]$, with $p$ a prime.

\begin{enumerate}
    \item $p\equiv 5\modulo{8}$ and $n=2$. Of the 43 occurrences of such an order in the table, 28 are locally associated.
    \item $p\not\equiv 3\modulo{4}$ and $q\equiv 3\modulo{4}$. Of the 50,139 occurrences of such an order in the table, 37,722 are locally associated.
    \item $p\not\equiv 3\modulo{4}$ and $n=q^2$ with $q\equiv 3\modulo{4}$, where $R_q$ is locally associated. Of the 865 occurrences of such an order in the table, 793 are locally associated.
\end{enumerate}

While these cases remain unsolved, it is worth noting that these are the only remaining unsolved cases. In particular, when determining whether an order $R_n$ in a real quadratic number field $K=\bQ[\sqrt{p}]$, orders which match the description in Conjecture \ref{pell's equation} or Case 1 above can be settled only by determining the fundamental unit in $K$. Orders which fall under Case 2 above can be settled by considering powers $u^k$ of the fundamental unit $u$ with $k|L(q,p)$. Finally, orders which fall under Case 3 above can be settled by noting whether $u^{L(q,p)}\in R_{q^2}$ (since $L(q^2,p)=q\cdot L(q,p)$). Whether any other order is locally associated can either be answered using the major results in the previous section or combining one of those results with one of the cases described here. This vastly cuts down on the work needed to identify locally associated orders.

\bibliographystyle{plain}
    \bibliography{bibliography}

@article{coykendallhfd, title={Half-factorial domains in quadratic fields}, volume={235}, DOI={10.1006/jabr.2000.8505}, number={2}, journal={Journal of Algebra}, author={Coykendall, Jim}, year={2001}, pages={417–430}}

@article{rago,
  title={A characterization of half-factorial orders in algebraic number fields},
  author={Rago, Balint},
  journal={arXiv preprint arXiv:2304.08099},
  year={2024}
}

@article{dissertation, title={Relating Elasticity and Other Multiplicative Properties Among Orders in Number Fields and Related Rings}, author={Moles, Grant}, year={2024}, journal={All Dissertations}, volume={3750}, note="\url{https://open.clemson.edu/all_dissertations/3750/}"}

@article{carlitz, title={A characterization of algebraic number fields with class number two}, volume={11}, DOI={10.2307/2034782}, number={3}, journal={Proceedings of the American Mathematical Society}, author={Carlitz, L.}, year={1960}, pages={391}}

@software{quadla,
author = {Moles, Grant},
license = {MIT},
month = mar,
title = {{Associated Quadratic Orders}},
version = {1.0.0},
year = {2025},
note="\url{https://github.com/gramoles/Associated-Quadratic-Orders}"
}

@misc{radicalconductor,
      title={Elasticity in orders of an algebraic number field with radical conductor ideal and their rings of formal power series}, 
      author={James Barker Coykendall and Grant Moles},
      year={2025},
note="\url{https://arxiv.org/abs/2505.01668}",
      eprint={2505.01668},
      archivePrefix={arXiv},
      primaryClass={math.AC},
      url={https://arxiv.org/abs/2505.01668}, 
}

@article{halter-koch, title={Factorization of Algebraic Integers}, volume={191}, journal={Ber. Math. Stat. Sektion Forschung}, author={Halter-Koch, Franz}, year={1983}}

@article{zaks1, title={Half factorial domains}, volume={82}, DOI={10.1090/s0002-9904-1976-14130-4}, number={5}, journal={Bulletin of the American Mathematical Society}, author={Zaks, Abraham}, year={1976}, month={Sep}, pages={721–723}}

@article{zaks2,
  title={Half-factorial-domains},
  author={Abraham Zaks},
  journal={Israel Journal of Mathematics},
  year={1980},
  volume={37},
  pages={281-302},
  url={https://api.semanticscholar.org/CorpusID:209832622}
}

@article{narkiewicz,
  title={A Note On Elasticity of Factorizations},
  author={Władysław Narkiewicz},
  journal={Journal of Number Theory},
  year={1995},
  volume={51},
  pages={46-47},
  url={https://api.semanticscholar.org/CorpusID:122200829}
}

@article{valenza,
  title={Elasticity of factorization in number fields},
  author={Robert J. Valenza},
  journal={Journal of Number Theory},
  year={1990},
  volume={36},
  pages={212-218},
  url={https://api.semanticscholar.org/CorpusID:121859013}
}

@misc{subringrelations,
      title={Multiplicative Relationships of Subrings and their Applications to Factorization}, 
      author={Grant Moles},
      year={2025},
note="\url{https://arxiv.org/abs/2506.24031}",
      eprint={2506.24031},
      archivePrefix={arXiv},
      primaryClass={math.AC},
      url={https://arxiv.org/abs/2506.24031}, 
}

@article{andersonandersonzafrullah,
title = {Factorization in integral domains},
journal = {Journal of Pure and Applied Algebra},
volume = {69},
number = {1},
pages = {1-19},
year = {1990},
issn = {0022-4049},
doi = {https://doi.org/10.1016/0022-4049(90)90074-R},
url = {https://www.sciencedirect.com/science/article/pii/002240499090074R},
author = {D.D. Anderson and David F. Anderson and Muhammad Zafrullah}
}

@article{boyntoncoykendall, title={On the Graph of Divisibility of an Integral Domain}, volume={58}, DOI={10.4153/CMB-2014-065-0}, number={3}, journal={Canadian Mathematical Bulletin}, author={Boynton, Jason Greene and Coykendall, Jim}, year={2015}, pages={449–458}}

@book{introalgnumthe, place={Cambridge}, title={Introductory Algebraic Number Theory}, publisher={Cambridge University Press}, author={Alaca, Saban and Williams, Kenneth S.}, year={2003}}
    
\end{document}